\documentclass[reqno,12pt]{amsart}
\usepackage{amsmath,amsfonts,amsthm,amssymb,indentfirst}
\usepackage[all,poly]{xy} 
\usepackage{color}
\usepackage[pagebackref,colorlinks]{hyperref}
\usepackage{kantlipsum}
\usepackage[all,poly]{xy}
\usepackage{mathrsfs}
\usepackage{enumerate}
\theoremstyle{plain}
\newtheorem{lemma}{Lemma}[section] 
\newtheorem{theorem}[lemma]{Theorem}
\newtheorem{corollary}[lemma]{Corollary}

\theoremstyle{definition}
\newtheorem{example}[lemma]{Example}
\newtheorem{definition}[lemma]{Definition}

\newcommand{\ol}{\overline}
\newcommand{\Por}{P_{(H,S)}}
\newcommand{\so}{\mathbf{s}}
\newcommand{\ra}{\mathbf{r}}

\title[Graded ideals of Leavitt path algebras]{Every graded ideal of a Leavitt path algebra is\\ {\em graded} isomorphic to a Leavitt path algebra}  

\author{Lia Va\v s}
\address{Department of Mathematics, Saint Joseph's University, Philadelphia, PA 19131, USA}
\email{lvas@sju.edu}



\thanks{The author is very grateful to the referee for the suggestion to consider formulating and proving the $C^*$-algebra version of the main result (now Corollary \ref{corollary_C_star}).}

\begin{document}
\begin{abstract} 
We show that every graded ideal of a Leavitt path algebra is {\em graded} isomorphic to a Leavitt path algebra. It is known that a graded ideal $I$ of a Leavitt path algebra is isomorphic to the Leavitt path algebra of a graph, known as the generalized hedgehog graph, which is defined based on certain sets of vertices uniquely determined by $I$. However, this isomorphism may not be graded. We show that replacing the short ``spines'' of the generalized hedgehog graph with possibly fewer, but then necessarily longer spines, we obtain a graph (which we call the porcupine graph) such that its Leavitt path algebra is {\em graded} isomorphic to $I$. Our proof adapts to show that for every closed
gauge-invariant ideal $J$ of a graph $C^*$-algebra, there is a gauge-invariant $*$-isomorphism mapping the graph $C^*$-algebra of the porcupine graph of $J$ onto $J.$ 
\end{abstract}

\maketitle

\section{Introduction}

If $E$ is a graph and $L_K(E)$ is its Leavitt path algebra, it is known that every graded ideal $I$ of $L_K(E)$ is uniquely determined by a pair $(H,S)$ of certain subsets of vertices, known as an admissible pair (we review the definition in section \ref{subsection_LPA_review}). For every such admissible pair, one can define a graph, referred to as the generalized hedgehog graph in \cite{LPA_book}, such that $I$ is isomorphic to the Leavitt path algebra of this graph. 
The name ``hedgehog'' comes from the construction in which one attaches new edges (the ``spines'') to $H\cup S$ (the ``body''). We modify this construction by replacing the added edges with possibly fewer paths but then necessarily of length larger than one. Because of the longer spines, we call the resulting graph the porcupine graph of $(H,S).$ In the main result of the paper, Theorem \ref{theorem_ideals}, we construct a {\em graded} $*$-isomorphism of $I$ and the Leavitt path algebra of the porcupine graph of the admissible pair corresponding to $I.$ 

We adapt our result to graph $C^*$-algebras also. In Corollary \ref{corollary_C_star}, we show that for every closed gauge-invariant ideal $J$ of a graph $C^*$-algebra there is a gauge-invariant (and graded in the $C^*$-algebra sense) $*$-isomorphism which maps the graph $C^*$-algebra of the porcupine graph of $J$ onto $J.$ Just as in the algebraic case, it is known that such ideal $J$ is $*$-isomorphic to a graph $C^*$-algebra, but it was not clear whether there is a gauge-invariant $*$-isomorphism between $J$ and a graph $C^*$-algebra. We show that such an isomorphism indeed exits.  

\section{Prerequisites}
\label{section_prerequisites}

\subsection{Graded rings and  \texorpdfstring{$\ast$}{TEXT}-rings prerequisites}\label{subsection_graded_rings_prerequisites}
A ring $R$ (not necessarily unital) is {\em graded} by a group $\Gamma$ if $R=\bigoplus_{\gamma\in\Gamma} R_\gamma$ for additive subgroups $R_\gamma$ and $R_\gamma R_\delta\subseteq R_{\gamma\delta}$ for all $\gamma,\delta\in\Gamma.$ The elements of the set $\bigcup_{\gamma\in\Gamma} R_\gamma$ are said to be {\em homogeneous} and $\gamma$ is the {\em degree} of any nonzero element of $R_\gamma.$ The grading is {\em trivial} if $R_\gamma=0$ for every $\gamma\in \Gamma$ which is not the group identity. We adopt the standard definitions of graded ring homomorphisms, graded algebras, and graded ideals as defined in \cite{Roozbeh_book}. 

A ring $R$ is an involutive ring or a $*$-ring, if there is an anti-automorphism $*:R\to R$ of order two. If $R$ is also a $K$-algebra for some commutative, involutive ring $K$, then $R$ is a $*$-algebra if $(kx)^*=k^*x^*$ for all $k\in K$ and $x\in R.$ If $R$ and $S$ are $*$-rings, a ring homomorphism $\phi: R\to S$ is a $*$-homomorphism if $\phi(x^*)=\phi(x)^*$ for every $x\in R.$

\subsection{Leavitt path algebras prerequisites}
\label{subsection_LPA_review} 

Let $E$ be a directed graph, let $E^0$ denote the set of vertices, $E^1$ the set of edges, and  $\so$ and $\ra$ denote the source and range maps of $E.$   A {\em sink} of $E$ is a vertex which emits no edges and an {\em infinite emitter} is a vertex which emits infinitely many edges. A vertex of $E$ is {\em regular} if it is not a sink or an infinite emitter. A {\em path} is a single vertex or a sequence of edges $e_1e_2\ldots e_n$ for some positive integer $n$ such that $\ra(e_i)=\so(e_{i+1})$ for $i=1,\ldots, n-1.$  

Extend a graph $E$ to the graph with the same vertices and with edges $E^1\cup \{e^*\mid e\in E^1\}$ where the range and source functions are the same as in $E$ for $e\in E^1$ and $\so(e^*)=\ra(e)$  and $\ra(e^*)=\so(e)$ for the added edges. 
If $K$ is any field, the \emph{Leavitt path algebra} $L_K(E)$ of $E$ over $K$ is a free $K$-algebra generated by the set  $E^0\cup E^1\cup\{e^\ast\mid e\in E^1\}$ such that for all vertices $v,w$ and edges $e,f,$

\begin{tabular}{ll}
(V)  $vw =0$ if $v\neq w$ and $vv=v,$ & (E1)  $\so(e)e=e\ra(e)=e,$\\
(E2) $\ra(e)e^\ast=e^\ast\so(e)=e^\ast,$ & (CK1) $e^\ast f=0$ if $e\neq f$ and $e^\ast e=\ra(e),$\\
(CK2) $v=\sum_{e\in \so^{-1}(v)} ee^\ast$ for each regular vertex $v.$ &\\
\end{tabular}

By the first four axioms, every element of $L_K(E)$ can be represented as a sum of the form $\sum_{i=1}^n k_ip_iq_i^\ast$ for some $n$, paths $p_i$ and $q_i$, and elements $k_i\in K,$ for $i=1,\ldots,n$ where $v^*=v$ for $v\in E^0$ and $p^*=e_n^*\ldots e_1^*$ for a path $p=e_1\ldots e_n.$ Using this representation, one can make $L_K(E)$ into an involutive ring by $\left(\sum_{i=1}^n k_ip_iq_i^\ast\right)^*=\sum_{i=1}^n k_i^*q_ip_i^\ast$ where $k_i\mapsto k_i^*$ is any involution on $K$. For more details on these basic properties, see \cite{LPA_book}.

If we consider $K$ to be trivially graded by $\mathbb Z,$ $L_K(E)$ is naturally graded by $\mathbb Z$ so that the $n$-component $L_K(E)_n$ is the $K$-linear span of the elements $pq^\ast$ for paths $p, q$ with $|p|-|q|=n$ where $|p|$ denotes the length of a path $p.$ While one can grade a Leavitt path algebra by any group (see \cite[Section 1.6.1]{Roozbeh_book}), we consider only the natural grading by $\mathbb Z.$ 

If a $K$-algebra $R$ contains elements $a_v,$ $b_e,$ and $c_{e^*}$ which satisfy the axioms (V), (E1), (E2), (CK1), and (CK2) where $v\in E^0$ and $e\in E^1$ for some graph $E$, the Universal Property of $L_K(E)$ ensures that the map $\phi: v\mapsto a_v, e\mapsto b_e, e^*\mapsto c_{e^*}$ has a unique $K$-algebra homomorphism extension $\phi: L_K(E)\to R$ (see \cite[Remark 1.2.5]{LPA_book}). If $R$ is $\mathbb Z$-graded and the elements $\phi(v),$ $\phi(e),$ and $\phi(e^*)$ have degrees 0, 1, and $-1$ respectively, such extension is graded and the Graded Uniqueness Theorem (\cite[Theorem 2.2.15]{LPA_book}) states that $\phi$ is injective if and only if $\phi(v)\neq 0$ for every $v\in E^0.$ 
If $R$ is involutive and if $a_v^*=a_v$ for every $v\in E^0$ and $(b_e)^*=c_{e^*}$ for every $e\in E^1,$ then $\phi$ is a $*$-homomorphism (see \cite[Lemma 4.7]{Lia_traces}).

\subsection{Graded ideals of a Leavitt path algebra}
A subset $H$ of $E^0$ is said to be {\em hereditary} if  $\ra(p)\in H$ for any path $p$ such that $\so(p)\in H.$ The set $H$ is {\em saturated} if $v\in H$ for any regular vertex $v$ such that $\ra(\so^{-1}(v))\subseteq H.$ We recall a construction from \cite{Tomforde}.
If $H$ is hereditary and saturated, let 
\[B_H=\{v\in E^0-H\,|\, v\mbox{ infinite emitter and }\so^{-1}(v)\cap \ra^{-1}(E^0-H)\mbox{ is nonempty and finite}\}\mbox{ and}\]
\[v^H=v-\sum_{e\in \so^{-1}(v)\cap \ra^{-1}(E^0-H)}ee^*\;\;\;\;\mbox{ for }v\in B_H.\]

An {\em admissible pair} is an ordered pair $(H, S)$ where $H\subseteq E^0$ is hereditary and saturated and $S\subseteq B_H.$ For such a pair, let $I(H,S)$ denote the graded ideal generated by homogeneous elements $H\cup \{v^H \,|\, v\in S \}.$ 
The ideal $I(H,S)$ is the $K$-linear span of the elements $pq^*$ for paths $p,q$ with $\ra(p)=\ra(q)\in H$ and the elements $pv^Hq^*$ for paths $p,q$ with $\ra(p)=\ra(q)=v\in S$ (see \cite[Lemma 5.6]{Tomforde}). The converse holds as well: for every graded ideal $I$, the vertices in $I$ form a hereditary and saturated set $H$ and the set of infinite emitters such that $v^H\in I$ is a subset of $B_H$ (\cite[Theorem 5.7]{Tomforde}, also \cite[Theorem 2.5.8]{LPA_book}). 
For an admissible pair $(H, S),$ let $E_{(H,S)}$ be the graph defined by
\[
\begin{array}{l}
F_1(H,S)=\{e_1\ldots e_n\mbox{ is a path of }E\mid \ra(e_n)\in H, \so(e_n)\notin H\cup S\}, \\
F_2(H,S)=\{p\mbox{ is a path of }E\mid \ra(p)\in S,\; |p|>0\},
\hskip.3cm \ol{F_i}(H,S)\mbox{ is a copy of }F_i(H,S), i=1,2,\\
E_{(H,S)}^0 = H\cup S\cup F_1(H,S)\cup F_2(H,S),\mbox{ and}\\
E_{(H,S)}^1 =\{e\in E^1\,|\, \so(e)\in H\}\cup \{e\in E^1\,|\, \so(e)\in S, \ra(e)\in H\}\cup \ol{F_1}(H,S)\cup \ol{F_2}(H,S)\mbox{ with }\so\mbox{ and }\ra\\
\mbox{the same as on $E^1$ for $e\in E^1\cap E^1_{(H,S)}$ and  }\so(\ol p)=p, \ra(\ol p)=\ra(p)\mbox{ for }\ol p\in\ol{F_1}(H,S)\cup \ol{F_2}(H,S).
\end{array} 
\]
By \cite[Theorem 6.1]{Ruiz_Tomforde} (also \cite[Theorem 2.5.22]{LPA_book}), the ideal $I(H,S)$ is isomorphic to $L_K (E_{(H,S)}).$

In \cite{LPA_book}, the graph $E_{(H,\emptyset)}$ is referred to as the {\em hedgehog graph} and 
the graph $E_{(H,S)}$ is referred to as the {\em generalized hedgehog graph}. The term ``hedgehog'' is chosen since the added edges ending at $H\cup S$ resemble spines of the hedgehog body $H\cup S.$ The isomorphism $L_K(E_{(H,S)})\to I(H,S)$ is defined so that the edges of $E_{(H,S)}$ are mapped to certain homogeneous elements of $I(H,S)$ of {\em positive} degree, not necessarily degree one. Thus, the degrees are not necessarily preserved under this map and so this isomorphism may not be graded. 

In \cite{Gonzalo_Ranga_Mercedes_selfinjective}, the authors consider a version of $E_{(H,S)}$ they denote $_HE_S$ (see \cite{Ruiz_Tomforde} for details on differences in definitions). \cite[Proposition 3.7]{Gonzalo_Ranga_Mercedes_selfinjective} states that $I(H,S)$ and $L_K(_HE_S)$ are {\em graded} isomorphic. However, the map $L_K(_HE_S)\to I(H,S)$ in the proof is defined analogously to the map $L_K(E_{(H,S)})\to I(H,S)$ and, as we noted above, this map is not necessarily a graded map. The last section of \cite{Ruiz_Tomforde} contains some further details on a problem with the definition of $_HE_S$ and \cite[Proposition 3.7]{Gonzalo_Ranga_Mercedes_selfinjective}.

\section{The main result}
\label{section_main}
We modify the construction of the generalized hedgehog graph of an admissible pair $(H,S)$ by making the spines added to the body $H\cup S$ possibly fewer in number but then necessarily longer in length. Because of the longer spines, we call the resulting graph the porcupine graph and denote it by $\Por.$ The modified construction enables us to create a degree preserving isomorphism $L_K(\Por)\to I(H,S).$ 

\begin{definition}
For an admissible pair $(H,S),$ we keep the definitions of $F_1(H,S)$ and $F_2(H,S).$
For each $e\in (F_1(H,S)\cup F_2(H,S))\cap E^1,$ let $w^e$ be a new vertex and $f^e$ a new edge such that $\so(f^e)=w^e$ and $\ra(f^e)=\ra(e).$
Continue this process inductively as follows. 
For each path $p=eq$ where $q\in F_1(H,S)\cup F_2(H,S)$ and $|q|\geq 1,$ add a new vertex $w^p$ and a new edge $f^p$ such that $\so(f^p)=w^p$ and $\ra(f^p)=w^q.$

We define the {\em porcupine graph} $\Por$ as follows. The set of vertices of $\Por$ is 
\[H\cup S\cup \{w^p \mid p\in F_1(H,S)\cup F_2(H,S)\}.\]
The set of edges of $\Por$  is
\[\{e\in E^1\,|\, \so(e)\in H\}\cup \{e\in E^1\,|\, \so(e)\in S, \ra(e)\in H\}\cup \{f^p\mid p\in F_1(H,S)\cup F_2(H,S)\}\]
The $\so$ and $\ra$ maps of $\Por$ are the same as in $E$ for the common edges and they are defined as above for the new edges.     
\end{definition}

Before formulating and proving the main result, we present examples comparing and contrasting the generalized hedgehog and the porcupine graphs of an admissible pair.

\begin{example}
Let $E$ be the Toeplitz graph $\;\;\xymatrix{{\bullet}^w\ar@(lu,ld)_e  \ar[r]^g & {\bullet}^v}$ and let $H=\{v\}.$ The hedgehog graph of $(H, \emptyset)$ is listed first and the porcupine graph second. 
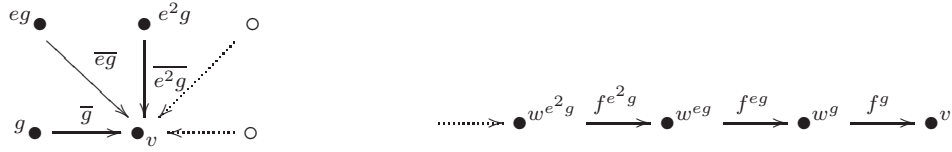
\[\xymatrix{^{eg}\bullet\ar[dr]^{\ol{eg}} &  \bullet\ar[d]^{\ol{e^2g}}  & \hskip-1.1cm^{e^2g}\hskip.7cm \circ \ar@{.>}[dl] \\
^g \bullet \ar[r]^{\ol g}  & \bullet_v   &  \circ \ar@{.>}[l] }
\hskip2cm \xymatrix{\\ \ar@{.>}[r] &
\bullet^{w^{e^2g}} \ar[r]^{f^{e^2g}} &  \bullet^{w^{eg}}\ar[r]^{f^{eg}}&  \bullet^{w^g}\ar[r]^{f^g} & \bullet^v }\]
The graded isomorphism of Theorem \ref{theorem_ideals} is such that the path $e^{n-1}g$ of length $n$ corresponds to the path $f^{e^{n-1}g}f^{e^{n-2}g}\ldots f^g$ of length $n$ also. 

We consider another example with an infinite emitter. Let $E$ be the graph $
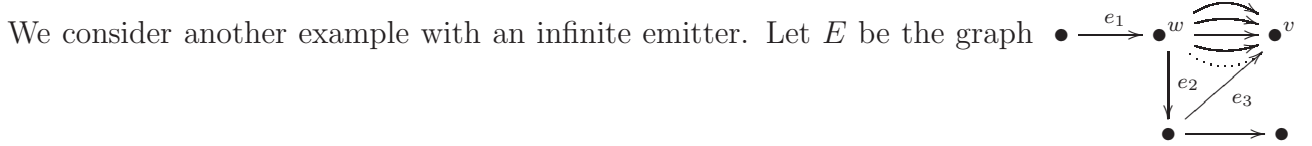
\xymatrix{
\bullet \ar[r]^{e_1} &{\bullet}^{w} \ar@{.} @/_1pc/ [r] _{\mbox{ } } \ar@/_/ [r] \ar [r] \ar@/^/ [r] \ar@/^1pc/ [r] \ar[d]^{e_2} 
& {\bullet}^{v}\\ & \bullet\ar[ur]_{e_3}\ar[r]&\bullet }$ and consider $H=\{v\}, S=\{w\}$. 
In this case, $F_1(H,S)=\{e_3, e_2e_3, e_1e_2e_3\} $ and $F_2(H,S)=\{e_1\}.$ The generalized hedgehog graph of $(H,S)$ is listed first and the porcupine graph second. 
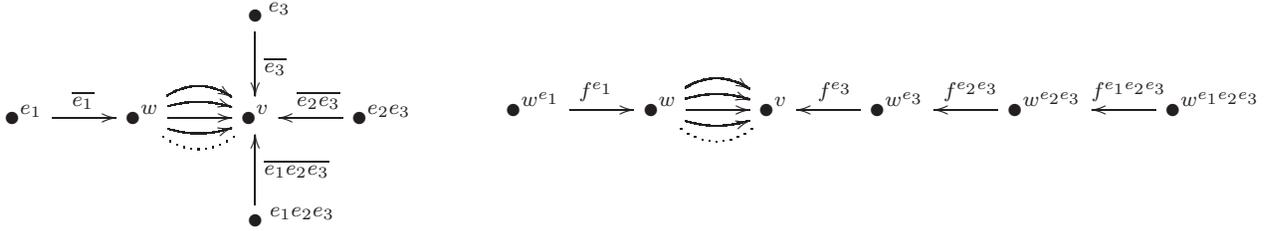
\[\xymatrix{&&\bullet\ar[d]^{\ol{e_3}}& \hskip-2.7cm ^{e_3} \\ \bullet^{e_1} \ar[r]^{\ol{e_1}} &{\bullet}^{w} \ar@{.} @/_1pc/ [r] _{\mbox{ } } \ar@/_/ [r] \ar [r] \ar@/^/ [r] \ar@/^1pc/ [r]  & {\bullet}^{v} & \bullet^{e_2e_3}\ar[l]_{\ol{e_2e_3}}\\ && \bullet\ar[u]_{\ol{e_1e_2e_3}}& \hskip-2.1cm^{e_1e_2e_3}}\hskip1cm 
\xymatrix{\\\bullet^{w^{e_1}} \ar[r]^{f^{e_1}} &{\bullet}^{w} \ar@{.} @/_1pc/ [r] _{\mbox{ } } \ar@/_/ [r] \ar [r] \ar@/^/ [r] \ar@/^1pc/ [r]  & {\bullet}^{v} & \bullet^{w^{e_3}}\ar[l]_{f^{e_3}}& \bullet^{w^{e_2e_3}}\ar[l]_{f^{e_2e_3}}&\bullet^{w^{e_1e_2e_3}}\ar[l]_{f^{e_1e_2e_3}} }\]
\end{example}

\begin{theorem}
For an admissible pair $(H,S)$ of a graph $E,$
\[I(H,S)\mbox{ and } L_K (\Por)\mbox{ are graded $*$-isomorphic.}\]
Thus, every graded ideal of a Leavitt path algebra is graded $*$-isomorphic to a Leavitt path algebra.
\label{theorem_ideals} 
\end{theorem}
\begin{proof}
We define a function $\phi$ which maps the vertices and edges of $\Por$ to elements of $ I(H,S)$ as follows. For a vertex $v$ of $\Por,$ we let  
\[
\phi(v)=\left\{\begin{array}{ll}
v & \text{if } v\in H\\
v^H & \text{if } v\in S\\
pp^* & \text{if } v=w^p\mbox{ and  }p\in F_1(H,S) \\
p\ra(p)^Hp^* & \text{if } v=w^p\mbox{ and  }p\in F_2(H,S)
\end{array}\right.\]
For an edge $g$ of $\Por,$ we let 
\[
\phi(g)=\left\{\begin{array}{ll}
e & \text{if } g=e\in E^1\mbox{ or if }g=f^e\mbox{ for some }e\in F_1(H,S)\cap E^1\\
e\ra(e)^H  & \text{if } g=f^e\;\; \mbox{ for some }e\in F_2(H,S)\cap E^1\\
epp^* & \text{if } g=f^{ep} \;\mbox{ for some  }e\in E^1\mbox{ and }p\in F_1(H,S)\\ 
ep\ra(p)^Hp^* & \text{if } g=f^{ep} \;\mbox{ for some  }e\in E^1\mbox{ and }p\in F_2(H,S)
\end{array}\right.\]

Extend $\phi$ to $g^*$ for $g\in\Por^1$ by $\phi(g^*)=\phi(g)^*$ and use definitions to check that (V), (E1), (E2), and (CK1) hold for $\phi(v), \phi(g), \phi(g^*), v\in \Por^0, g\in \Por^1.$ We present more details for checking that (CK2) holds. If $v$ is a regular vertex of $\Por,$ then $v$ cannot be in $S,$ so $v\in H$ or $v=w^p$ for some $p\in F_1(H,S)\cup F_2(H,S).$ In the first case, the edges $v$ emits are in $E$ since $H$ is hereditary and the relation $v=\sum_{e\in \so^{-1}(v)} ee^*$ holds in $L_K(\Por)$ since it holds in $L_K(E).$ In the second case, consider the four possible cases: $p=e\in F_1(H,S)\cap E^1,$  $p=e\in F_2(H,S)\cap E^1,$ $p=eq$ for $e\in E^1, q\in F_1(H,S),$ and $p=eq$ for $e\in E^1, q\in F_2(H,S).$ Note that in each case, $v$ emits only one edge $f^p$ so it is sufficient to check that 
$\phi(f^p)\phi((f^p)^*)=\phi(v)$ in each case. 

If $p=e\in F_1(H,S)\cap E^1,$ then $\phi(f^e)\phi((f^e)^*)=ee^*=\phi(v).$ 

If $p=e\in F_2(H,S)\cap E^1,$ then $\phi(f^e)\phi((f^e)^*)=e\ra(e)^H \ra(e)^H e^*=e\ra(e)^H e^*=\phi(v).$

If $p=eq$ for $e\in E^1, q\in F_1(H,S),$ then $\phi(f^p)\phi((f^p)^*)=eqq^*qq^*e^*=eqq^*e^*=pp^*=\phi(v).$ 

If $p=eq$ for $e\in E^1, q\in F_2(H,S),$ then $\phi(f^p)\phi((f^p)^*)=eq\ra(p)^Hq^*q\ra(p)^Hq^*e^*=
eq\ra(p)^Hq^*e^*=p\ra(p)^Hp^*=\phi(v).$
This shows that (CK2) holds. 

By the Universal Property, $\phi$ has a unique extension to a $K$-algebra homomorphism $L_K(\Por)\to I(H,S).$ Since $\phi$ preserves degrees of $v\in \Por^0$ and of $g$ and $g^*$ for $g\in \Por^1,$ this extension, which we denote also by $\phi,$ is a {\em graded} homomorphism. As $\phi(v)^*=\phi(v)$ for $v\in \Por^0$ and $\phi(g^*)=\phi(g)^*$ for $g\in \Por^1,$ $\phi$ is a $*$-homomorphism. The map $\phi$ is nonzero on every vertex of $\Por,$ so $\phi$ is injective by the Graded Uniqueness Theorem.

Since $\phi$ is a $*$-homomorphism, to show surjectivity of $\phi,$ it is sufficient to show (1) and (2) where (1) denotes the condition that $p$ is in the image of $\phi$ for every path such that $\ra(p)\in H,$ and (2) denotes the condition that $p\ra(p)^H$ is in the image of $\phi$ for every path $p$ such that $\ra(p)\in S.$ Both claims hold for paths of zero length since $\phi(v)=v$ if $v\in H$ and $\phi(v)=v^H$ if $v\in S.$ Thus, consider a path $p=e_1\ldots e_n$ for a positive integer $n.$ 

To show (1) assume that $\ra(p)\in H.$ If $\so(p)\in H,$ then $\phi(p)=p$ so the claim holds. If $\so(p)\notin H,$ let $i\in \{1,2,\ldots, n\}$ be the largest such that $\so(e_i)\notin H.$ We consider three cases: (i) $\so(e_i)\notin S,$ (ii)  $\so(e_i)\in S$ and $i=1,$ and (iii) $\so(e_i)\in S$ and $i>1.$ 

In case (i), $e_j\ldots e_i\in F_1(H,S)$ for all $j=1,\ldots i$ and we have that 
\[\phi(f^{e_1\ldots e_i}f^{e_2\ldots e_i}\ldots f^{e_i} e_{i+1}\ldots e_n)=e_1\ldots e_i(e_2\ldots e_i)^*(e_2\ldots e_i)(e_3\ldots e_i)^*\;\ldots\; e_{i-1}e_ie_i^*e_i \;\;e_{i+1}\ldots e_n=\]\[=e_1\ldots e_i\;e_{i+1}\ldots e_n=e_1\ldots e_n=p\]
if $i<n$ and the analogous argument applies to  the case $i=n.$  

In case (ii), $\phi(e_j)=e_j$ for every $j=1,\ldots, n$ and $\phi(p)=p.$ 

In case (iii), $e_j\ldots e_{i-1}\in F_2(H,S)$ for every $j=1,\ldots i-1$ and $\ra(e_{i-1})^He_i=e_i$ so that
\[\phi(f^{e_1\ldots e_{i-1}}f^{e_2\ldots e_{i-1}} \ldots f^{e_{i-1}}e_i\ldots e_n)=e_1\ldots e_{i-1}\ra(e_{i-1})^H(e_2\ldots e_{i-1})^*
e_2\ldots e_{i-1}\ra(e_{i-1})^H\ldots \]\[\ldots 
e_{i-1}^*e_{i-1}\ra(e_{i-1})^He_ie_{i+1}\ldots e_n=
e_1\ldots e_{i-1}\ra(e_{i-1})^He_ie_{i+1}\ldots e_n=
e_1\ldots e_{i-1}e_ie_{i+1}\ldots e_n=p.\]
This shows that (1) holds. 

To show (2), assume that $\ra(p)\in S.$ Then $e_i\ldots e_n$ is in $F_2(H,S)$ for every $i=1,\ldots, n$ and 
\[\phi(f^{e_1\ldots e_n}f^{e_2\ldots e_n}\ldots f^{e_n})=e_1\ldots e_n\ra(p)^H(e_2\ldots e_n)^*(e_2\ldots e_n)\ra(p)^H(e_3\ldots e_n)^*\;\ldots\]\[\ldots e_{n-1}e_n\ra(p)^He_n^*e_n\ra(p)^H=e_1\ldots e_n\ra(p)^H=p\ra(p)^H.\]
This shows that $\phi$ is surjective. 

The second sentence of the theorem is a direct corollary of the first sentence and  \cite[Theorem 5.7]{Tomforde} (also \cite[Theorem 2.5.8]{LPA_book}).  
\end{proof}

\subsection{Graph  \texorpdfstring{$\mathbf{C^*}$}{TEXT}-algebras}
\label{subsection_C_star}

Theorem \ref{theorem_ideals} has its graph $C^*$-algebra version. If $E$ is a graph, the {\em graph $C^*$-algebra of $E$} is the universal
$C^*$-algebra generated by mutually orthogonal projections $\{p_v\mid v\in E^0\}$ and
partial isometries with mutually orthogonal ranges $\{s_e\mid e\in E^1\}$ satisfying the analogues of the (CK1) and (CK2) axioms and the axiom (CK3) stating that $s_es_e^*\leq p_{\so(e)}$ for every $e\in E^1.$ The term ``universal'' in the definition means that the $C^*$-algebra version of the algebraic Universal Property, mentioned before,  holds (see \cite[Definition 5.2.5]{LPA_book}). By letting $s_{e_1\ldots e_n}$ be $s_{e_1}\ldots s_{e_n}$ and $s_v=p_v$ for $e_1,\ldots ,e_n\in E^1$ and $v\in E^0,$ $s_p$ is defined for every path $p.$

The set $\{p_v, s_e\mid v\in E^0, e\in E^1\}$ is referred to as a {\em Cuntz-Krieger $E$-family}. 
For such an $E$-family and an element $z$ of the unit
circle $\mathbb T$, one defines a map $\gamma^E_z$ by $\gamma^E_z(p_v)=p_v$ and $\gamma^E_z(s_e)=zs_e$   and then uniquely extends this map to a $*$-automorphism of $C^*(E)$ (we assume a homomorphism of a $C^*$-algebra to be bounded). The {\em gauge action} $\gamma^E$ on $\mathbb T$ is given by $\gamma^E(z)=\gamma^E_z.$ Note that $\gamma^E_z(s_ps_q^*)=z^{|p|-|q|}s_ps_q^*$ for $z\in \mathbb T$ and paths $p$ and $q.$ The presence of the degree $|p|-|q|$ of $pq^*$ in the previous formula explains the connection of this action and the $\mathbb Z$-grading of $L_{\mathbb C}(E)$.
If the integral of a continuous function $f:\mathbb T\to \mathbb C$ over $\mathbb T$ is defined by  $\int_{\mathbb T}f(z)dz=\int_0^1 f(e^{2\pi i t})dt,$
the gauge action on $C^*(E)$ determines a $\mathbb Z$-grading of $C^*(E)$ (see \cite[Proposition 5.2.11]{LPA_book}) so that \[C^*(E)_n=\{x\in C^*(E)\mid \int_{\mathbb T}z^{-n}\gamma^E_z(x)dz=x\}\label{*}\tag{*}\]
is the completion of $L_{\mathbb C}(E)_n$ and $C^*(E)$ is the completion of $\bigoplus_{n\in\mathbb Z} C^*(E)_n.$ We note that this grading is not a grading in the algebraic sense, but in the $C^*$-algebra sense (see the paragraph following \cite[Theorem 5.2.9.]{LPA_book}). 

A closed ideal $J$ of a graph $C^*$-algebra $C^*(E)$ is {\em gauge-invariant} if $\gamma^E_z(J)=J$ for every $z\in \mathbb T.$ By \cite[Theorem 3.6]{Bates_et_al}, every such ideal $J$ is the completion of the linear span
of the elements $s_ps_q^*$ for paths $p,q$ with $\ra(p)=\ra(q)\in H$ and the elements $s_pp_v^Hs_q^*$ for paths $p,q$ with $\ra(p)=\ra(q)=v\in S$ where $p_v^H=p_v-\sum_{e\in \so^{-1}(v)\cap \ra^{-1}(E^0-H)}s_es_e^*$ for $v\in B_H$ and where $(H, S)$ is the admissible pair uniquely determined by $J$ and defined analogously as for an ideal of $L_K(E)$. Also as before, an admissible pair $(H,S)$ uniquely determines a closed gauge-invariant ideal $J(H,S)$ which is graded in the $C^*$-algebra sense.  

If $R$ is a $C^*$-algebra with an action $\beta: \mathbb T\to$ Aut$(R),$ we say that a $*$-homomorphism $\phi: C^*(E)\to R$ is {\em gauge-invariant} if $\beta_z\circ \phi=\phi\circ \gamma^E_z$ for every $z\in\mathbb T.$ Thus, if $J$ is a closed gauge-invariant ideal of $C^*(E),$ $F$ is a graph, and $\phi: C^*(F)\to J$ is a $*$-homomorphism, $\phi$ is gauge-invariant if $\gamma_z^E|_J\circ \phi=\phi\circ \gamma^F_z$ for every $z\in\mathbb T.$ It is direct to see that every such gauge-invariant map $\phi$ is graded. Indeed, using the formula (\ref{*}) and the fact that $\phi$ is bounded, one directly checks that $\phi(C^*(F)_n)\subseteq J_n.$

\begin{corollary}
For an admissible pair $(H,S)$ of a graph $E,$
there is a gauge-invariant (thus graded) $*$-isomorphism 
\[\phi: C^*(\Por)\cong J(H,S).\]
Thus, for every closed gauge-invariant ideal $J$ of a graph $C^*$-algebra, there is a gauge-invariant (thus graded) $*$-isomorphism mapping the graph $C^*$-algebra of the porcupine graph corresponding to $J$ onto $J.$ 
\label{corollary_C_star} 
\end{corollary}
\begin{proof}
One defines a map $\phi$ on $\{p_v, s_g\mid v\in \Por^0, g\in \Por^1\}$ analogously as in Theorem \ref{theorem_ideals}. We claim that the image of this map constitutes a Cuntz-Krieger $\Por$-family. Indeed, $\{\phi(p_v)\mid v\in \Por^0\}$ is a set of orthogonal projections in $J(H,S)$ and $\{\phi(s_g)\mid g\in \Por^1\}$ is a set of partial isometries in $I(H,S)$ with mutually orthogonal ranges. One checks that (CK1) and (CK2) hold in the same way as it was done in the proof of Theorem \ref{theorem_ideals}. If $g\in\Por^1$ and $\so(g)$ is regular, the requirement $\phi(s_g)\phi(s_g)^*\leq \phi(p_{\so(g)})$ follows from (CK2). If $g\in\Por^1$ and $\so(g)$ is an infinite emitter, $\so(g)\in H\cup S$ and $g$ is an edge in $E^1.$ Thus, (CK3) holds for $s_g$ in $C^*(\Por)$ since it holds for $s_g$ in $C^*(E).$  
  
The map $\phi$ has a unique extension to a $*$-homomorphism $\phi:C^*(\Por)\to J(H,S)$ by the Universal Property. To show that $\phi$ is gauge-invariant, it is sufficient to check that the condition $\gamma_z^E|_I\circ \phi=\phi\circ \gamma^{\Por}_z$ holds on $\{p_v, s_g\mid v\in \Por^0, g\in \Por^1\}$ which is direct to check using definitions. 

The map $\phi$ is nonzero on every element $p_v$ for $v\in \Por^0,$ so $\phi$ is injective by the Gauge-Invariant Uniqueness Theorem (see \cite[Theorem 5.2.12]{LPA_book}). The surjectivity of $\phi$ holds by arguments analogous to those in the proof of Theorem \ref{theorem_ideals}.
\end{proof}


\begin{thebibliography}{10}
\bibitem{LPA_book} G. Abrams, P. Ara, M. Siles Molina, Leavitt path algebras, Lecture Notes in Mathematics 2191, Springer, London, 2017. 

\bibitem{Gonzalo_Ranga_Mercedes_selfinjective}
G.  Aranda Pino, K. M. Rangaswamy, M.  Siles Molina, \emph{Weakly Regular and Self-Injective Leavitt Path
Algebras Over Arbitrary Graphs}, Algebr. Represent. Theory {\bf 14} (2011), 751--777. 

\bibitem{Bates_et_al}  T. Bates, J. H. Hong, I. Raeburn, W. Szyma\'{n}ski, \emph{The ideal structure of the $C^*$-algebras of infinite graphs}, Illinois J. Math. {\bf 46 (4)} (2002), 1159--1176. 

\bibitem{Roozbeh_book} R. Hazrat, Graded rings and graded Grothendieck groups, London Math. Soc. Lecture Note Ser. 435, Cambridge Univ. Press, 2016.

\bibitem{Ruiz_Tomforde} E. Ruiz, M. Tomforde, \emph{Ideals in Graph Algebras,} Algebr. Representat. Theory {\bf 17 (3)} (2014), 849--861.

\bibitem{Tomforde} M. Tomforde, \emph{Uniqueness theorems and ideal structure for Leavitt path algebras} J. Algebra {\bf 318 (1)} (2007), 270--299.

\bibitem{Lia_traces} L. Va\v s, \emph{Canonical traces and directly finite Leavitt path algebras}, Algebr. Represent. Theory {\bf 18} (2015), 711--738. 


\end{thebibliography}
\end{document}